\documentclass[a4paper,10pt]{amsart}
\usepackage{latexsym}
\usepackage{amssymb}
\usepackage[all]{xy}
\newtheorem{theorem}{Theorem}[section]
\newtheorem{lemma}[theorem]{Lemma}
\newtheorem{proposition}[theorem]{Proposition}
\newtheorem{corollary}[theorem]{Corollary}
\theoremstyle{definition}

\theoremstyle{remark}
\newtheorem{remark}[theorem]{Remark}

\newcommand{\Mod}{\text{\rm mod-}}
\newcommand{\Hom}{\operatorname{Hom}}

\newcommand{\Ext}{\operatorname{Ext}}
\newcommand{\End}{\operatorname{End}}

\title{On the cardinalities of Kronecker quiver Grassmannians}

\author{Csaba Sz\'ant\'o}

\begin{document}
\maketitle \pagestyle{myheadings} \markboth{\sc Csaba
Sz\'ant\'o}{\sc On the cardinalities of some Kronecker quiver
Grassmannians}

{\bf Abstract.} We deduce using the Ringel-Hall algebra approach
explicit formulas for the cardinalities of some Grassmannians over a
finite field associated to the Kronecker quiver. We realize in this
way a quantification of the formulas obtained by Caldero and
Zelevinsky for the Euler characteristics of these Grassmannians. We
also present a recursive algorithm for computing the cardinality of
every Kronecker quiver Grassmannian over a finite field.

\medskip


\medskip


\bigskip

\begin{center}{\bf Introduction}\end{center}
\medskip

Let $K$ be the Kronecker quiver $K: \xymatrix{1 & 2 \ar@<1ex>
[l]^{\beta} \ar@<-1ex>[l]_{\alpha}}$, $kK$ be the Kronecker algebra
over the finite field $k=\Bbb F_q$ with $q$ elements and $\Mod kK$
the category of its finite dimensional right modules (called
Kronecker modules). We consider the rational Ringel-Hall algebra
$\mathcal{H}(kK)$ of the Kronecker algebra, with a $\Bbb Q$-basis
formed by the isomorphism classes $[M]$ from $\Mod kK$ and
multiplication
$$[N_1][N_2]=\sum_{[M]}F^M_{N_1N_2}[M].$$ The structure constants
$F^M_{N_1N_2}=|\{M\supseteq U|\ U\cong N_2,\ M/U\cong N_1\}|$ are
called Ringel-Hall numbers.

For any module $M\in\Mod kK$, and any $\underline{e}=(a,b)$ in $\Bbb
N^2$, we denote by $Gr_{\underline{e}}(M)_{\Bbb F_q}$ the
Grassmannian of submodules of $M$ with dimension vector
$\underline{e}$:
$$Gr_{\underline{e}}(M)_{\Bbb F_q} = \{N\in\Mod kK|N\leq M, \underline{dim}(N)
=\underline{e}\}.$$ Then we have that
$$|Gr_{\underline{e}}(M)_{\Bbb F_q}|=\sum_{\begin{tiny}\begin{array}{c}[X],[Y]\\\underline{dim}Y=\underline{e}\end{array}\end{tiny}}F^M_{XY}$$

The Grassmannian cardinalities above play an important role in the
theory of cluster algebras. In \cite{CalRei} Caldero and Reineke
show for affine quivers that these cardinal numbers are given by
integral polynomials in $q$ with positive coefficients. So in our
case there is an integral polynomial $p_{\underline{e},M}$ such that
$|Gr_{\underline{e}}(M)_{\Bbb F_q}|=p_{\underline{e},M}(q)$.
Moreover the Euler characteristics $\chi(Gr_{\underline{e}}(M)_{\Bbb
C})=p_{\underline{e},M}(1)$.

In \cite{CalZel} Caldero and Zelevinsky describe explicit
combinatorial formulas for the Euler characteristics
$\chi(Gr_{\underline{e}}(M)_{\Bbb C})=p_{\underline{e},M}(1)$
whenever $M$ is indecomposable.

Using specific recursions obtained by the Ringel-Hall algebra
approach and the use of reflection functors we deduce in this paper
explicit combinatorial formulas for the cardinalities (polynomials)
$|Gr_{\underline{e}}(M)_{\Bbb F_q}|=p_{\underline{e},M}(q)$ whenever
$M$ is indecomposable. We obtain in this way a quantification of the
formulas by Caldero and Zelevinsky. Moreover our recursions provide
a recursive algorithm for computing the cardinality of every
Kronecker quiver Grassmannian over a finite field.
\section{\bf Facts on Kronecker modules and Ringel-Hall algebras}

The indecomposables in $\Mod kK$ are divided into three families:
the preprojectives, the regulars and the preinjectives (see
\cite{assem},\cite{aus},\cite{rin1}).

The preprojective (respectively preinjective) indecomposable modules
are up to isomorphism uniquely determined by their dimension
vectors. For $n\in\Bbb N$ we will denote by $P_n$ (respectively with
$I_n$) the indecomposable preprojective module of dimension
$(n+1,n)$ (respectively the indecomposable preinjective module of
dimension $(n,n+1)$). So $P_0$, $P_1$ are the projective
indecomposable modules ($P_0=S_1$ being simple) and $I_0$, $I_1$ the
injective indecomposable modules ($I_0=S_2$ being simple).

The regular indecomposables (up to isomorphism) are $R_p(t)$ for
$t\geq 1$ and $p\in\Bbb P^1_k$ of dimension vector $(td_p,td_p)$
($d_p$ standing for the degree of the point $p$). The module
$R_p(t)$ has regular length $t$ and regular socle the regular simple
$R_p(1)$. Suppose that $R_p(0)=0$. Note that $R_p(t)$ is regular
uniserial meaning that the only regular submodule series of $R_p(t)$
is $0\subset R_p(1)\subset...\subset R_p(t)$.

We will denote by $R_p(\lambda)$ (where $\lambda$ is a partition)
the module $\oplus_iR_p(\lambda_i)$ and by $P$ (respectively
$I$,$R$) a module with all its indecomposable components
preprojective (respectively preinjective, regular).

Denote by $cM=M\oplus...\oplus M$ $c$-times.

The following lemma is well known.
\begin{lemma}

{\rm a)} $\Hom ({R},{P})=\Hom ({I},{P})=\Hom ({I},{R})=
\Ext^1({P},{R})=\Ext^1({P},{I})= \Ext^1({R},{I})=0.$

{\rm b)} There are no nontrivial morphisms and extensions between
regular modules from different tubes, i.e. if $p\neq p'$, then
$\Hom(R_p(t),R_{p'}(t'))=\Ext^1(R_p(t),R_{p'}(t'))=0.$

{\rm c)} For $n\leq m$, we have $\dim _k\Hom (P_n,P_m)=m-n+1$ and
$\Ext^1(P_n,P_m)=0$; otherwise $\Hom (P_n,P_m)=0$ and $\dim
_k\Ext^1(P_n,P_m)=n-m-1$. In particular $\End(P_n)\cong k$ and
$\Ext^1(P_n,P_n)=0$.

{\rm d)} For $n\geq m$, we have $\dim _k\Hom (I_n,I_m)=n-m+1$ and
$\Ext^1(I_n,I_m)=0$; otherwise $\Hom (I_n,I_m)=0$ and $\dim
_k\Ext^1(I_n,I_m)=m-n-1$. In particular $\End(I_n)\cong k$ and
$\Ext^1(I_n,I_n)=0$.

{\rm e)} $\dim _k\Hom (P_n,I_m)=n+m$ and $\dim
_k\Ext^1(I_m,P_n)=m+n+2$.

{\rm f)} $\dim _k\Hom (P_n,R_p(t))=\dim _k\Hom (R_p(t),I_n)=d_pt$
and $\dim _k\Ext^1(R_p(t),P_n)=\dim _k\Ext^1(I_n,R_p(t))=d_pt$.

{\rm g)} $\dim _k\Hom (R_p(t_1),R_p(t_2))=\dim
_k\Ext^1(R_p(t_1),R_p(t_2))=d_p\min{(t_1,t_2)}$.

\end{lemma}

Let now $\tilde K$ be the quiver obtained by reversing the arrows in
$K$. The category $\Mod k\tilde K$ can be identified with the
category $\Mod kK$ after a formal relabeling of the vertices. In
general we will denote by $\neg M\in\Mod k\tilde K$ (respectively by
$\neg M\in\Mod kK$) the relabeled version of $M\in\Mod kK$
(respectively of $M\in\Mod k\tilde K$). So we have $\neg{\neg M}=M$.

For $i=1,2$ denote by $\Mod kK\langle i\rangle$ (respectively by
$\Mod k\tilde K\langle i\rangle$) the full subcategory of modules
not containing the simple $S_i$ (respectively the simple $\tilde
S_i$) as a direct summand. Notice that using the formal relabeling
mentioned above $\Mod k\tilde K\langle 1\rangle$ can be identified
with $\Mod kK\langle 2\rangle$ and $\Mod k\tilde K\langle 2\rangle$
with $\Mod kK\langle 1\rangle$. Since the vertex 1 is a sink in the
quiver $K$ and a source in $\tilde K$ the restriction of the
corresponding reflection functors to the above mentioned full
subcategories will give us the following inverse pair of
equivalences
$$\neg R^+_1:\Mod kK\langle 1\rangle\to\Mod kK\langle 2\rangle
\text{ and }R^-_1\neg:\Mod kK\langle 2\rangle\to\Mod kK\langle
1\rangle.$$

We refer to \cite{dlabrin} for all notions and properties related to
the reflection functors. Notice that we have $$\neg R^+_1(P_n)=
P_{n-1}\text{ for } n\in\Bbb N^*, \neg R^+_1(I_n)=I_{n+1}\text{ for
} n\in\Bbb N.$$ Moreover for each $d$ there is a permutation
$\sigma_d$ of the set $\{p\in\Bbb P^1_k|d_p=d\}$ such that for each
$p$ with $d_p=d$ we have
$$\neg R^+_1(R_p(t))= R_{\sigma_d(p)}(t)\text{ for all } t\in\Bbb
N^*.$$
\medskip
\begin{remark} Notice that if $M\in\Mod kK\langle 1\rangle$ with
$\underline{dim}M=(m,n)$, then $m-n$ is at most $n$ and
$\underline{dim}(\neg R^+_1(M))=(m-(m-n),n-(m-n))=(n,2n-m)$. Also if
$M\in\Mod kK\langle 2\rangle$ with $\underline{dim}M=(m,n)$, then
$n-m$ is at most $m$ and
$\underline{dim}(R^-_1\neg(M))=(m+(m-n),n+(m-n))=(2m-n,m)$.
\end{remark}
Related with the Ringel-Hall algebra we will need the following
properties (see \cite{Szanto}):

\begin{lemma} (Associativity of the Ringel-Hall multiplication)
$\sum_{[X]}F^X_{AB}F^M_{XC}=\sum_{[X]}F^M_{AX}F^X_{BC}$.
\end{lemma}

\begin{lemma} For $N_1,N_2\in \Mod kK$ with
$\Ext^1(N_1,N_2)=0$ and $\Hom (N_2,N_1)=0$ we have
$[N_1][N_2]=[N_1\oplus N_2]$.
\end{lemma}

\begin{lemma} {\rm a)} If
$M$,$N$ and $L$ are in $\Mod kK\langle 1\rangle$, then
$F^L_{MN}=F^{\neg R^+_1(L)}_{\neg R^+_1(M)\neg R^+_1(N)}$.

{\rm b)} If $M$,$N$ and $L$ are in $\Mod kK\langle 2\rangle$, then
$F^L_{MN}=F^{R^-_1\neg(L)}_{R^-_1\neg(M)R^-_1\neg(N)}$.
\end{lemma}

\section{\bf Identities for Gaussian coefficients}
For $l,a\in\Bbb Z$, $l>0$ we will denote by
$G^l_a(q)=\frac{(q^a-1)...(q^{a-l+1}-1)}{(q^l-1)...(q-1)}$ the
Gaussian (q-binomial) coefficients. By definition $G^0_a(q)=1$ and
$G^{-l}_a(q)=0$. The following properties of the Gaussian
coefficients are well known

\begin{lemma} {\rm a)} $G^l_a(q)=0$ for $0\leq a<l$. Also
$G^l_a=G^{a-l}_a$ for $a,l\geq 0$.

{\rm b)} (Cross product) For all $a,l,j\in\Bbb Z$ we have
$G^l_a(q)G^j_l(q)=G^j_a(q)G^{l-j}_{a-j}(q)$.

{\rm c)} (q-Vandermonde) For all $l,a,b\in\Bbb Z$ we have
$G^l_{a+b}(q)=\sum_{j\in\Bbb
Z}q^{j(a-l+j)}G^{l-j}_a(q)G^j_b(q)=\sum_{r\in\Bbb
Z}q^{(l-r)(a-r)}G^{r}_a(q)G^{l-r}_b(q)$. Notice that the sums are
finite.
\end{lemma}

Finally we will prove a q-analogue of the so called Nanjundiah
identity (see \cite{Nan})

\begin{proposition} For all $m,p,\mu,\nu\in\Bbb Z$ we have
$$\sum_{r\in\Bbb Z}
q^{(m-\mu+\nu-r)(p-r)}G^r_{m-\mu+\nu}(q)G^{p-r}_{p+\mu-\nu}(q)G^{m+p}_{\mu+r}(q)=G^m_{\mu}(q)G^p_{\nu}(q)$$
\end{proposition}

\begin{proof} Denote by $A$ the left expression and by $B$ the right one.

One can immediately see that for $p<0$ we have $A=B=0$.

Applying 3 times Lemma 2.1. c) and  2 times Lemma 2.1. b) we have
$$A=\sum_{r\in\Bbb Z}
q^{(m-\mu+\nu-r)(p-r)}G^r_{m-\mu+\nu}(q)G^{p-r}_{p+\mu-\nu}(q)G^{m+p}_{\mu+r}(q)$$$$=\sum_r
q^{(m-\mu+\nu-r)(p-r)}G^r_{m-\mu+\nu}(q)G^{p-r}_{p+\mu-\nu}(q)\sum_{s\in\Bbb
Z} q^{s(\mu-m-p+s)} G^{m+p-s}_{\mu}(q)G^s_r(q)$$$$=\sum_{r,s\in\Bbb
Z}q^{s(\mu-m-p+s)}
q^{(m-\mu+\nu-r)(p-r)}G^r_{m-\mu+\nu}(q)G^s_r(q)G^{p-r}_{p+\mu-\nu}(q)
G^{m+p-s}_{\mu}(q)$$$$=\sum_{r,s\in\Bbb Z}q^{s(\mu-m-p+s)}
q^{(m-\mu+\nu-r)(p-r)}G^s_{m-\mu+\nu}(q)G^{r-s}_{m-\mu+\nu-s}(q)G^{p-r}_{p+\mu-\nu}(q)
G^{m+p-s}_{\mu}(q)$$$$=\sum_{s\in\Bbb
Z}q^{s(\mu-m-p+s)}G^s_{m-\mu+\nu}(q)G^{m+p-s}_{\mu}(q)\sum_{r\in\Bbb
Z}
q^{(m-\mu+\nu-r)(p-r)}G^{r-s}_{m-\mu+\nu-s}(q)G^{p-r}_{p+\mu-\nu}(q)$$
$$=\sum_{s\in\Bbb Z}q^{s(\mu-m-p+s)}G^s_{m-\mu+\nu}(q)G^{m+p-s}_{\mu}(q)\sum_{t\in\Bbb
Z}
q^{(m-\mu+\nu-s-t)(p-s-t)}G^{t}_{m-\mu+\nu-s}(q)G^{p-s-t}_{p+\mu-\nu}(q)$$$$=\sum_{s\in\Bbb
Z}q^{s(\mu-m-p+s)}G^s_{m-\mu+\nu}(q)G^{m+p-s}_{\mu}(q)G^{p-s}_{m+p-s}(q)$$
$$=\sum_{s\in\Bbb Z}q^{s(\mu-m-p+s)}G^s_{m-\mu+\nu}(q)G^{p-s}_{\mu}(q)G^{m}_{\mu-p+s}(q)$$

One can see from here that for $m<0$ we have $A=0$ and trivially
also $B=0$.

Consider now the case $m,p\geq 0$. Then using Lemma 2.1. a),b)
notice that
$$A=\sum_{s\in\Bbb Z}q^{s(\mu-m-p+s)}G^s_{m-\mu+\nu}(q)G^{m+p-s}_{\mu}(q)G^{p-s}_{m+p-s}(q)$$
$$=\sum_{s=0}^pq^{s(\mu-m-p+s)}G^s_{m-\mu+\nu}(q)G^{m+p-s}_{\mu}(q)G^{p-s}_{m+p-s}(q)$$
$$=\sum_{s=0}^pq^{s(\mu-m-p+s)}G^s_{m-\mu+\nu}(q)G^{m+p-s}_{\mu}(q)G^{m}_{m+p-s}(q)$$
$$=\sum_{s\in\Bbb Z}q^{s(\mu-m-p+s)}G^s_{m-\mu+\nu}(q)G^{m+p-s}_{\mu}(q)G^{m}_{m+p-s}(q)$$
$$=\sum_{s\in\Bbb Z}q^{s(\mu-m-p+s)}G^s_{m-\mu+\nu}(q)G^{m}_{\mu}(q)G^{p-s}_{\mu-m}(q)$$
$$=G^{m}_{\mu}(q)\sum_{s\in\Bbb Z}q^{s(\mu-m-p+s)}G^s_{m-\mu+\nu}(q)G^{p-s}_{\mu-m}(q)$$
$$=G^{m}_{\mu}(q)G^{p}_{\nu}(q)=B$$
\end{proof}
\section{\bf The recursions}
Let $a,b\in\Bbb Z$. We introduce the following notations.

For $M\in\Mod kK$
$$A^M_{a,b}:=|Gr_{(a,b)}(M)|=\sum_{\begin{tiny}\begin{array}{c}[X],[Y]\\\underline{dim}Y=(a,b)\end{array}\end{tiny}}F^M_{XY}$$

For $M\in \Mod kK\langle 1\rangle$
$$B^M_{a,b}:=\sum_{\begin{tiny}\begin{array}{c}[X],[Y]\\\underline{dim}Y=(a,b)\\X,Y\in \Mod kK\langle 1\rangle\end{array}\end{tiny}}F^M_{XY}$$

For $M\in \Mod kK\langle 2\rangle$
$$C^M_{a,b}:=\sum_{\begin{tiny}\begin{array}{c}[X],[Y]\\\underline{dim}Y=(a,b)\\X,Y\in \Mod kK\langle 2\rangle\end{array}\end{tiny}}F^M_{XY}$$

The sums $A^M_{a,b}$, $B^M_{a,b}$, $C^M_{a,b}$ are considered to be
0 if they are empty. In particular they are 0 if $a<0$ or $b<0$.
Also notice that if $\underline{dim}M=(m,n)$ then
$A^M_{a,b}=B^M_{a,b}=C^M_{a,b}=0$ for $a>m$ or $b>n$.

\begin{proposition} Suppose up to isomorphism $M=sS_1\oplus M'\oplus tS_2$ with $M'\in\Mod kK\langle 1\rangle\cap\Mod kK\langle
2\rangle$. Let $a,b\in\Bbb Z$, $l=a-b$ and $\underline{dim}M=(m,n)$.

{\rm a)} We have that
$$A^M_{a,b}=\sum_{c\in\Bbb Z} G^c_{m-a+c}(q)B^{M'\oplus tS_2}_{a-c,b}=\sum_{c\in\Bbb Z}
G^c_{m-a+c}(q)C^{\neg R^+_1(M'\oplus tS_2)}_{a-l,b-l+c},$$ the sum
being finite.

{\rm b)} We have that
$$A^M_{a,b}=\sum_{d\in\Bbb Z} G^d_{b+d}(q)C^{sS_1\oplus M'}_{a,b+d}=\sum_{d\in\Bbb Z} G^d_{b+d}(q)B^{R^-_1\neg(sS_1\oplus M')}_{a+l-d,b+l},$$ the sum being finite.
\end{proposition}
\begin{proof} a) If $b<0$ then trivially $$A^M_{a,b}=\sum_{c\in\Bbb Z} G^c_{m-a+c}(q)B^{M'\oplus tS_2}_{a-c,b}=\sum_{c\in\Bbb Z}
G^c_{m-a+c}(q)C^{\neg R^+_1(M'\oplus tS_2)}_{b,b-l+c}=0.$$ If $a<0$
then $A^M_{a,b}=0$. The sum $\sum_{c\in\Bbb Z}
G^c_{m-a+c}(q)B^{M'\oplus tS_2}_{a-c,b}=0$ because for $c<0$
$G^c_{m-a+c}(q)=0$ and for $c\geq 0$ we have $a-c<0$ so $B^{M'\oplus
tS_2}_{a-c,b}=0$. We also have $\sum_{c\in\Bbb Z}
G^c_{m-a+c}(q)C^{\neg R^+_1(M'\oplus tS_2)}_{b,2b-a+c}=0$ because
for $c<0$ $G^c_{m-a+c}(q)=0$ and for $c\geq 0$ there is no $Y\in\Mod
kK\langle 2\rangle$ with dimension $(b,2b-a+c)$ (see Remark 1.2.) so
$C^{\neg R^+_1(M'\oplus tS_2)}_{a-l,b-l+c}=0$.

Consider now the case $a,b\geq 0$. Firstly notice that if $Y_c\in
\Mod kK\langle 1\rangle$ then by Lemma 1.1. and Lemma 1.4. we have
$[cS_1][Y_c]=[cS_1\oplus Y_c]$, so $F^{Z}_{cS_1 Y_c}=1$ for
$[Z]=[cS_1\oplus Y_c]$ and $F^{Z}_{cS_1 Y_c}=0$ in all the other
cases.

Using Lemma 1.3. we obtain:
$$A^M_{a,b}=\sum_{\begin{tiny}\begin{array}{c}[X],[Y]\\\underline{dim}Y=(a,b)\end{array}\end{tiny}}F^M_{XY}=\sum_{c\geq 0}\sum_{\begin{tiny}\begin{array}{c}[X],[Y_c]\\\underline{dim}Y_c=(a-c,b)\\Y_c\in
\Mod kK\langle 1\rangle\end{array}\end{tiny}}F^M_{X cS_1\oplus
Y_c}\cdot 1$$$$=\sum_{c\geq
0}\sum_{\begin{tiny}\begin{array}{c}[X],[Y_c],[Z]\\\underline{dim}Y_c=(a-c,b)\\Y_c\in
\Mod kK\langle 1\rangle\end{array}\end{tiny}}F^M_{X Z}\cdot
F^{Z}_{cS_1 Y_c}=\sum_{c\geq
0}\sum_{\begin{tiny}\begin{array}{c}[X],[Y_c],[Z]\\\underline{dim}Y_c=(a-c,b)\\Y_c\in
\Mod kK\langle 1\rangle\end{array}\end{tiny}}F^Z_{X\ cS_1}\cdot
F^{M}_{Z Y_c}$$$$=\sum_{c\geq
0}\sum_{\begin{tiny}\begin{array}{c}[Y_c],[Z]\\\underline{dim}Y_c=(a-c,b)\\\underline{dim}Z=(m-a+c,n-b)
\\Y_c\in \Mod kK\langle
1\rangle\end{array}\end{tiny}}(\sum_{[X]}F^Z_{X\ cS_1})\cdot
F^{M}_{Z Y_c}=\sum_{c\geq
0}G^c_{m-a+c}(q)\sum_{\begin{tiny}\begin{array}{c}[Y_c],[Z]\\\underline{dim}Y_c=(a-c,b)
\\Y_c\in \Mod kK\langle 1\rangle\end{array}\end{tiny}}F^{M}_{Z
Y_c}$$$$=\sum_{c\geq
0}G^c_{m-a+c}(q)\sum_{\begin{tiny}\begin{array}{c}[Y_c],[Z']\\\underline{dim}Y_c=(a-c,b)
\\Y_c,Z'\in \Mod kK\langle 1\rangle\end{array}\end{tiny}}F^{M'\oplus
tS_2}_{Z' Y_c}=\sum_{c\in\Bbb Z} G^c_{m-a+c}(q)B^{M'\oplus
tS_2}_{a-c,b}$$

Here we have used the following: if $F^M_{ZY_c}=F^{sS_1\oplus
M'\oplus tS_2}_{ZY_c}\neq 0$ then since $Y_c\in\Mod kK\langle
1\rangle$ it follows that $Y_c$ embeds only in $M'\oplus tS_2$ (see
Lemma 1.1.), so $Z=sS_1\oplus Z'$ with $0\to Y_c\to M'\oplus tS_2\to
Z'\to 0$ exact, $Z'\in\Mod kK\langle 1\rangle$ (because $M'\oplus
tS_2$ does not project on $S_1$) and in this way $F^{sS_1\oplus
M'\oplus tS_2}_{ZY_c}=F^{M'\oplus tS_2}_{Z'Y_c}$.

To prove the other identity we will use reflection functors. Using
Remark 1.2. and  Lemma 1.5. we have $$A^M_{a,b}=\sum_{c\geq
0}G^c_{m-a+c}(q)\sum_{\begin{tiny}\begin{array}{c}[Y_c],[Z']\\\underline{dim}Y_c=(a-c,b)
\\Y_c,Z'\in \Mod kK\langle 1\rangle\end{array}\end{tiny}}F^{M'\oplus
tS_2}_{Z' Y_c}$$$$=\sum_{c\geq
0}G^c_{m-a+c}(q)\sum_{\begin{tiny}\begin{array}{c}[Y_c],[Z']\\\underline{dim}Y_c=(a-c,b)
\\Y_c,Z'\in \Mod kK\langle 1\rangle\end{array}\end{tiny}}F^{\neg
R^+_1(M'\oplus tS_2)}_{\neg R^+_1(Z') \neg
R^+_1(Y_c)}$$$$=\sum_{c\geq
0}G^c_{m-a+c}(q)\sum_{\begin{tiny}\begin{array}{c}[Y'_c],[Z'']\\\underline{dim}Y'_c=(a-l,b-l+c)
\\Y'_c,Z''\in \Mod kK\langle 2\rangle\end{array}\end{tiny}}F^{\neg
R^+_1(M'\oplus tS_2)}_{Z''Y'_c}$$$$=\sum_{c\in\Bbb Z}
G^c_{m-a+c}(q)C^{\neg R^+_1(M'\oplus tS_2)}_{a-l,b-l+c}.$$

b) dual of a).
\end{proof}

We can state now the recursion theorem for the numbers $A^M_{a,b}$.

\begin{theorem} Suppose up to isomorphism $M=sS_1\oplus M'\oplus tS_2$ with $M'\in\Mod kK\langle 1\rangle\cap\Mod kK\langle
2\rangle$. Let $a,b\in\Bbb Z$, $l=a-b$ and $\underline{dim}M=(m,n)$.
We have the following recursions

a) $$A^M_{a,b}=\sum_{c\in\Bbb Z} q^{c(b-l+c)}G^c_{m-2b}(q)A^{\neg
R^+_1(M'\oplus tS_2)}_{a-l,b-l+c},$$ the sum being finite.

b) $$A^M_{a,b}=\sum_{d\in\Bbb Z}
q^{d(2m-n+t-a-l+d)}G^d_{2a-2m+n-t}(q)A^{R^-_1\neg(sS_1\oplus
M')}_{a+l-d,b+l},$$ the sum being finite.
\end{theorem}
\begin{proof} a) Using the previous proposition and the fact that $\neg R^+_1(M'\oplus tS_2)\in\Mod kK\langle
2\rangle$ we have
$$A^M_{a,b}=\sum_{c\in\Bbb Z}
G^c_{m-a+c}(q)C^{\neg R^+_1(M'\oplus tS_2)}_{a-l,b-l+c},$$
$$A^{\neg R^+_1(M'\oplus tS_2)}_{a-l,b-l+c}=\sum_{d\in\Bbb Z}
G^d_{b-l+c+d}(q)C^{\neg R^+_1(M'\oplus tS_2)}_{a-l,b-l+c+d}.$$ Let
$u=c+d$. Using Lemma 2.1. c)
$$\sum_{c\in\Bbb Z} q^{c(b-l+c)}G^c_{m-2b}(q)A^{\neg
R^+_1(M'\oplus tS_2)}_{a-l,b-l+c}=\sum_{c,d\in\Bbb Z}
q^{c(b-l+c)}G^c_{m-2b}(q)G^d_{b-l+c+d}(q)C^{\neg R^+_1(M'\oplus
tS_2)}_{a-l,b-l+c+d}$$$$=\sum_{u\in\Bbb Z}(\sum_{c\in\Bbb Z}
q^{c(b-l+c)}G^c_{m-2b}(q)G^{u-c}_{b-l+u}(q))C^{\neg R^+_1(M'\oplus
tS_2)}_{a-l,b-l+u}=\sum_{u\in\Bbb Z} G^u_{m-a+u}(q)C^{\neg
R^+_1(M'\oplus tS_2)}_{a-l,b-l+u}=A^M_{a,b}.$$

b)Using the previous proposition, Remark 1.2. and the fact that
$R^-_1\neg(sS_1\oplus M')\in\Mod kK\langle 1\rangle$ we have
$$A^M_{a,b}=\sum_{d\in\Bbb Z} G^d_{b+d}(q)B^{R^-_1\neg(sS_1\oplus M')}_{a+l-d,b+l},$$
$$A^{R^-_1\neg(sS_1\oplus
M')}_{a+l-d,b+l}=\sum_{c\in\Bbb Z}
G^c_{2m-n+t-a-l+d+c}(q)B^{R^-_1\neg(sS_1\oplus M')}_{a+l-d-c,b+l}.$$
Let $u=c+d$. Using Lemma 2.1. c)
$$\sum_{d\in\Bbb Z}
q^{d(2m-n+t-a-l+d)}G^d_{2a-2m+n-t}(q)A^{R^-_1\neg(sS_1\oplus
M')}_{a+l-d,b+l}$$$$=\sum_{c,d\in\Bbb Z}
q^{d(2m-n+t-a-l+d)}G^d_{2a-2m+n-t}(q)G^c_{2m-n+t-a-l+d+c}(q)B^{R^-_1\neg(sS_1\oplus
M')}_{a+l-d-c,b+l}$$$$=\sum_{u\in\Bbb Z}(\sum_{d\in\Bbb Z}
q^{d(2m-n+t-a-l+d)}G^d_{2a-2m+n-t}(q)G^{u-d}_{2m-n+t-a-l+u}(q))B^{R^-_1\neg(sS_1\oplus
M')}_{a+l-u,b+l}$$$$=\sum_{u\in\Bbb Z}
G^u_{b+u}(q)B^{R^-_1\neg(sS_1\oplus M')}_{a+l-u,b+l}=A^M_{a,b}.$$
\end{proof}

\section{\bf Formulas for the cardinalities $A^{M}_{a,b}=|Gr_{(a,b)}(M)|$ with $M$
indecomposable}

Using the recurrences from the previous section we will provide
closed formulas for $A^{P_n}_{a,b},A^{I_n}_{a,b}$ (with $n\in\Bbb
N$, $a,b\in\Bbb Z$) and $A^{R_p(t)}_{a,b}$ (with $t\in\Bbb N^*$,
$a,b\in\Bbb Z$ and $p\in\Bbb P^1_k$ of degree 1).

\begin{theorem} $A^{P_n}_{a,b}=|Gr_{(a,b)}(P_n)|=\left\{\begin{array}{cc} 0 & \text{for $a<0$ or
$b<0$}\\1 & \text{for $a=b=0$}\\
G^{n+1-a}_{n+1-b}(q)G^{a-b-1}_{a-1}(q)
&\text{otherwise}\end{array}\right.$
\end{theorem}
\begin{remark} Using the definitions and Lemma 2.1. a) notice that $G^{n+1-a}_{n+1-b}(q)G^{a-b-1}_{a-1}=0$
for $0<a\leq b$, for $a>{n+1}$, for $b>n$, for $a>0$ and $b<0$.
\end{remark}
\begin{proof} Induction on $n$. For $n=0$ we have that $A^{P_0}_{a,b}=1$ when $(a,b)=(1,0)$ or $(a,b)=(0,0)$ and $0$
otherwise so using the previous remark we can see that the formula
is true.

Suppose now $n\geq 1$. Then trivially $A^{P_n}_{a,b}=0$ for $a<0$ or
$b<0$ and $A^{P_n}_{0,0}=1$ so we only need to look at the case
$a,b\geq 0$, $a^2+b^2\neq 0$. Using Theorem 3.2. a) we obtain the
recursion
$$A^{P_n}_{a,b}=\sum_{c\in\Bbb
Z}q^{c(b-l+c)}G^c_{n-2b+1}(q)A^{P_{n-1}}_{a-l,b-l+c},$$ the sum
being finite.

Using Remark 4.2. and the induction hypothesis notice that if $b>0$
(and $a\geq 0$) then
$$A^{P_{n-1}}_{a-l,b-l+c}=A^{P_{n-1}}_{b,2b-a+c}=G^{n-b}_{n-2b+a-c}(q)G^{a-b-c-1}_{b-1}(q)$$
so denoting by $u=a-b-c-1$, using the previous recursion and
Proposition 2.2. with the entries $p=a-b-1$, $m=n+1-a$, $\mu=n+1-b$
and $\nu=a-1$
$$A^{P_n}_{a,b}=\sum_{c\in\Bbb
Z}q^{c(b-l+c)}G^c_{n-2b+1}(q)A^{P_{n-1}}_{a-l,b-l+c}$$$$=\sum_{c\in\Bbb
Z}q^{c(2b-a+c)}G^c_{n-2b+1}(q)G^{n-b}_{n-2b+a-c}(q)G^{a-b-c-1}_{b-1}(q)$$
$$=\sum_{u\in\Bbb
Z}q^{(a-b-u-1)(b-u-1)}G^{a-b-u-1}_{n-2b+1}(q)G^{n-b}_{n-b+1+u}(q)G^{u}_{b-1}(q)=G^{n+1-a}_{n+1-b}(q)G^{a-b-1}_{a-1}(q)$$

If now $b=0$ (and $n+1\geq a>0$) then trivially
$$A^{P_n}_{a,b}=\sum_{\begin{tiny}\begin{array}{c}[X]\end{array}\end{tiny}}F^{P_n}_{X\text{
}{aS_1}}=G^{a}_{n+1}(q)=G^{n+1-a}_{n+1-0}(q)G^{a-0-1}_{a-1}(q).$$

If $b=0$ and $a>n+1$ then trivially $A^{P_n}_{a,b}=0$ (see Remark
4.2.).
\end{proof}
\begin{theorem}$A^{I_n}_{a,b}=|Gr_{(a,b)}(I_n)|=\left\{\begin{array}{cc} 0 & \text{for $a>n$ or
$b>n+1$}\\1 &\text{for
 $a=n$, $b=n+1$}\\
G^{a-b}_{n-b}(q)G^{b}_{a+1}(q) &\text{otherwise}\end{array}\right.$
\end{theorem}
\begin{remark} Using the definitions and Lemma 2.1. a) notice that $G^{a-b}_{n-b}(q)G^{b}_{a+1}=0$
for $a<b$, for $a<0$, for $b<0$, for $a>n$ and $b<n+1$.
\end{remark}
\begin{proof} Induction on $n$. For $n=0$ we have that $A^{I_0}_{a,b}=1$ when $(a,b)=(0,1)$ or $(a,b)=(0,0)$ and $0$
otherwise so using the previous remark we can see that the formula
is true.

Suppose now $n\geq 1$. Then trivially $A^{I_n}_{a,b}=0$ for $a>n$ or
$b>n+1$ and $A^{I_n}_{n,n+1}=1$ so we only need to look at the case
$a\leq n$, $b\leq n+1$ (with no simultaneous equality). Using
Theorem 3.2. b) we obtain the recursion
$$A^{I_n}_{a,b}=\sum_{d\in\Bbb
Z}q^{d(n-1-a-l+d)}G^d_{2a-n+1}(q)A^{I_{n-1}}_{a+l-d,b+l},$$ the sum
being finite.

Using Remark 4.4. and the induction hypothesis notice that if $a<
n$(and $b\leq n+1$) then
$$A^{I_{n-1}}_{a+l-d,b+l}=A^{I_{n-1}}_{2a-b-d,a}=G^{a-b-d}_{n-a-1}(q)G^{a}_{2a-b-d+1}(q)$$
so denoting by $u=a-b-d$, using the previous recursion and
Proposition 2.2. with the entries $p=a-b$, $m=b$, $\mu=a+1$ and
$\nu=n-b$
$$A^{I_n}_{a,b}=\sum_{d\in\Bbb
Z}q^{d(n-1-a-l+d)}G^d_{2a-n+1}(q)A^{I_{n-1}}_{a+l-d,b+l}$$$$=\sum_{d\in\Bbb
Z}q^{d(n-1-2a+b+d)}G^d_{2a-n+1}(q)G^{a-b-d}_{n-a-1}(q)G^{a}_{2a-b-d+1}(q)$$$$=\sum_{u\in\Bbb
Z}q^{(a-b-u)(n-1-a-u)}G^{a-b-u}_{2a-n+1}(q)G^{u}_{n-a-1}(q)G^{a}_{a+u+1}(q)=G^{a-b}_{n-b}(q)G^{b}_{a+1}(q).$$

If now $a=n$ (and $0\leq b<n+1$) then trivially
$$A^{I_n}_{a,b}=\sum_{\begin{tiny}\begin{array}{c}[X]\end{array}\end{tiny}}F^{I_n}_{(n+1-b)S_2\text{
}{X}}=G^{n+1-b}_{n+1}(q)=G^{n-b}_{n-b}(q)G^{b}_{n+1}(q).$$

If $a=n$ and $b<0$ then trivially $A^{I_n}_{a,b}=0$ (see Remark
4.4.).
\end{proof}
\begin{lemma}Let $t\in\Bbb N^*$,
$a,b\in\Bbb Z$ and $p\in\Bbb P^1_k$ of degree 1. Then we have

{\rm a)} $A^{R_p(t)}_{a,a}=1$ for $0\leq a\leq t$.

{\rm b)} $A^{R_p(t)}_{a,b}=0$ for $0\leq a<b\leq t$.

{\rm c)} For two points $p,p'\in\Bbb P^1_k$ of the same degree $1$
we have, that $A^{R_p(t)}_{a,b}=A^{R_{p'}(t)}_{a,b}$.
\end{lemma}
\begin{proof} a) Suppose $\underline{dim}Y=(a,a)$, $a>0$ (so the defect is
0) and $Y$ embeds into $R_p(t)$. Then using Lemma 1.1. and the
uniseriality of the regulars one can see that $Y$ must be of the
form $R_p(t')$ with $0<t'\leq t$. So it follows that for $0<a\leq
n$, we have $A^{R_p(t)}_{a,a}=F^{R_p(t)}_{R_p(t'')R_p(t')}=1$. The
rest of the statement follows easily.

b) If for $0\leq a<b\leq t$ $A^{R_p(t)}_{a,b}>0$ this would mean
that there is a module $Y$ of dimension $(a,b)$ which embeds into
$R_p(t)$. But $a<b$ means that $Y$ must have a preinjective
component. Using Lemma 1.1. one can notice that we can't embed a
preinjective into a regular module.

c) Using Lemma 1.1. and the uniseriality of regulars, observe that
for $F^{R_p(t)}_{XY}\neq 0$ the modules $X,Y$ can contain at most a
single regular direct component which is of the form $R_p(t')$.
Permuting the points $\{p\in \Bbb P^1_k|d_p=d\}$ the assertion
follows.
\end{proof}
\begin{theorem} Let $t\in\Bbb N^*$,
$a,b\in\Bbb Z$ and $p\in\Bbb P^1_k$ of degree 1. Then we have
$$A^{R_p(t)}_{a,b}=|Gr_{(a,b)}(R_p(t))|=\left\{\begin{array}{cc} 0 &
\text{for $a<0$ or $b<0$}\\G^{t-a}_{t-b}(q)G^{a-b}_{a}(q) &
\text{otherwise}\end{array}\right.$$
\end{theorem}
\begin{remark} Using the definitions and Lemma 2.1. a) notice that $G^{t-a}_{t-b}(q)G^{a-b}_{a}=0$
for $a<b$, for $a>t$, for $b>t$ and $G^{t-a}_{t-b}(q)G^{a-b}_{a}=1$
for $0\leq a=b\leq t$.
\end{remark}
\begin{proof} Using Remark 4.7. observe that the formula is trivially true whenever $a<0$ or $b<0$ or $a>t$ or $b>t$.
Also when $b=0$ and $0\leq a\leq t$ then trivially
$$A^{R_p(t)}_{a,0}=\sum_{\begin{tiny}\begin{array}{c}[X]\end{array}\end{tiny}}F^{R_p(t)}_{X\text{
}{aS_1}}=G^{a}_{t}(q)=G^{t-a}_{t-0}(q)G^{a-0}_{a}(q).$$ Using Lemma
4.5. one can see that the formula is true in the cases $0\leq
a=b\leq t$ and $0\leq a<b\leq t$.

So we only need to consider the case $0<b<a\leq t$. Using Theorem
3.2. a) and Lemma 4.5. c) we obtain the recursion
$$A^{R_p(t)}_{a,b}=\sum_{c\in\Bbb
Z}q^{c(b-l+c)}G^c_{t-2b}(q)A^{R_p(t)}_{a-l,b-l+c},$$ the sum being
finite.

We proceed by induction on $a$. Using the recursion and the
considerations above for $a=2\leq t$ we have
$$A^{R_p(t)}_{2,1}=\sum_{c\in\Bbb
Z}q^{c^2}G^c_{t-2}(q)A^{R_p(t)}_{1,c}=G^0_{t-2}G^1_t+qG^1_{t-2}=G^{t-2}_{t-1}G^1_2$$

Let now $3\leq a\leq t$ and $0<b<a$. Using Remark 4.7. and the
induction hypothesis notice that
$$A^{R_p(t)}_{a-l,b-l+c}=A^{R_p(t)}_{b,2b-a+c}=G^{t-b}_{t-2b+a-c}(q)G^{a-b-c}_{b}(q)$$
so denoting by $u=a-b-c$, using the previous recursion and
Proposition 2.2. with the entries $p=a-b$, $m=t-a$, $\mu=t-b$ and
$\nu=a$
$$A^{R_p(t)}_{a,b}=\sum_{c\in\Bbb
Z}q^{c(b-l+c)}G^c_{t-2b}(q)A^{R_p(t)}_{a-l,b-l+c}$$$$=\sum_{c\in\Bbb
Z}q^{c(2b-a+c)}G^c_{t-2b}(q)G^{t-b}_{t-2b+a-c}(q)G^{a-b-c}_{b}(q)$$$$=\sum_{u\in\Bbb
Z}q^{(a-b-u)(b-u)}G^{a-b-u}_{t-2b}(q)G^{t-b}_{t-b+u}(q)G^{u}_{b}(q)=G^{t-a}_{t-b}(q)G^{a-b}_{a}$$
\end{proof}

We can see that in the cases above $|Gr_{\underline{e}}(M)_{\Bbb
F_q}|$ is an integer polynomial $p_{\underline{e},M}(q)$. Using that
$\chi(Gr_{\underline{e}}(M)_{\Bbb C})=p_{\underline{e},M}(1)$ and
$G^n_a(1)=\left(\begin{array}{c} a\\n\end{array}\right)$ we obtain
\begin{corollary}\cite{CalZel}

{\rm a)} $\chi(Gr_{(a,b)}(P_n)_{\Bbb C})=\left\{\begin{array}{cc} 0
& \text{for $a<0$ or
$b<0$}\\1 & \text{for $a=b=0$}\\
\left(\begin{array}{c}
n+1-b\\n+1-a\end{array}\right)\left(\begin{array}{c}
a-1\\a-b-1\end{array}\right) &\text{otherwise}\end{array}\right.$

{\rm b)} $\chi(Gr_{(a,b)}(I_n)_{\Bbb C})=\left\{\begin{array}{cc} 0
& \text{for $a>n$ or $b>n+1$}\\1 &\text{for
 $a=n$, $b=n+1$}\\\left(\begin{array}{c}
n-b\\a-b\end{array}\right)\left(\begin{array}{c}
a+1\\b\end{array}\right) &\text{otherwise}\end{array}\right.$

{\rm c)} $\chi(Gr_{(a,b)}(R_p(t))_{\Bbb C})=\left\{\begin{array}{cc}
0 & \text{for $a<0$ or $b<0$}\\\left(\begin{array}{c}
t-b\\t-a\end{array}\right)\left(\begin{array}{c}
a\\a-b\end{array}\right) & \text{otherwise}\end{array}\right.$
\end{corollary}
\begin{remark} Notice that there is no closed formula for $A^{R_p(t)}_{a,b}$ with $t\in\Bbb N^*$,
$a,b\in\Bbb Z$ and $p\in\Bbb P^1_k$ of degree $d_p>1$. This because
$A^{R_p(t)}_{a,a}=1$ only for $0\leq a\leq d_pt$ with $d_p|a$.
However this case will not appear over $\Bbb C$.
\end{remark}

\section{\bf A recursive algorithm for the cardinalities $A^{M}_{a,b}=|Gr_{(a,b)}(M)|$ with $M$
arbitrary}
Let $M\in\Mod kK$ arbitrary and suppose $\underline{dim}M=(m,n)$. We
know that up to isomorphism $M=P\oplus R\oplus I$ where $P$
(respectively $I$,$R$) is a module with all its indecomposable
components preprojective (respectively preinjective, regular). We
also know that $A^M_{a,b}=0$ for $a<0$ or $b<0$ or $a>m$ or $b>n$.

Applying the recursion from Theorem 3.2. a) after a finite number of
steps $A^{P\oplus R\oplus I}_{a,b}$ is reduced to knowing some
numbers of the form $A^{R'\oplus I'}_{a',b'}$. Applying the
recursion from Theorem 3.2. b) after a finite number of steps
$A^{R'\oplus I'}_{a',b'}$ is reduced to knowing some numbers of the
form $A^{R''}_{a'',b''}$. Using the arguments from the proof of
Lemma 4.5. b) we can see that $A^{R''}_{a'',b''}=0$ for $a''<b''$ so
applying the recursion from Theorem 3.2. a) after a finite number of
steps $A^{R''}_{a'',b''}$ with $a''\geq b''\geq 0$ is reduced to
knowing some numbers of the form $A^{R'''}_{a''',a'''}$. (Here
"some" means of course "a finite number").

Suppose $R'''=\oplus_{i=1}^mR_{p_i}(\lambda^i)$, where $\lambda^i$
are partitions, $p_i\in \Bbb P^1_k$ different points with degree
$d_{p_i}$ and $\sum_{i=1}^md_{p_i}|\lambda^i|=n$ so
$\underline{dim}R'''=(n,n)$. Denote $a'''$ simply by $a$ and suppose
$0\leq a\leq n$. For partitions $\lambda,\mu,\nu$ we will denote by
$g^{\lambda}_{\nu\mu}(q^{d_p})=F^{R_p(\lambda)}_{R_p(\nu)R_p(\mu)}$
the classical Hall polynomial (see \cite{Macd} for details). We know
that $g^{\lambda}_{\nu\mu}=g^{\lambda}_{\mu\nu}$ and
$g^{\lambda}_{\nu\mu}=0$ unless $|\lambda|=|\mu|+|\nu|$ and
$\mu,\nu\subseteq\lambda$.

Using Lemma 1.1. b) and Lemma 1.4. we have that
$$[\oplus_{i=1}^mR_{p_i}(\nu^i)][\oplus_{i=1}^mR_{p_i}(\mu^i)]=\prod_{i=1}^m[R_{p_i}(\nu^i)][R_{p_i}(\mu^i)]$$
so
$$F^{\oplus_{i=1}^mR_{p_i}(\lambda^i)}_{\oplus_{i=1}^mR_{p_i}(\nu^i)\ \oplus_{i=1}^mR_{p_i}(\mu^i)}=\prod_{i=1}^mF^{R_{p_i}(\lambda^i)}_
{R_{p_i}(\nu^i)R_{p_i}(\mu^i)}$$

Using the considerations above and the arguments from the proof of
Lemma 4.5. a) we will have
$$A^{R'''}_{a,a}=\sum_{\begin{tiny}
\begin{array}{c}\nu^i,\mu^i\subseteq\lambda^i\\\sum_{i=1}^md_{p_i}|\mu^i|=a\\\sum_{i=1}^md_{p_i}|\nu^i|=n-a\end{array}\end{tiny}}
F^{\oplus_{i=1}^mR_{p_i}(\lambda^i)}_{\oplus_{i=1}^mR_{p_i}(\nu^i)\
\oplus_{i=1}^mR_{p_i}(\mu^i)}=\sum_{\begin{tiny}
\begin{array}{c}\nu^i,\mu^i\subseteq\lambda^i\\\sum_{i=1}^md_{p_i}|\mu^i|=a\\\sum_{i=1}^md_{p_i}|\nu^i|=n-a\end{array}\end{tiny}}
\prod_{i=1}^mg^{\lambda^i}_ {\nu^i\mu^i}(q^{d_{p_i}})$$


\begin{thebibliography}{99}
\bibitem{assem} {\it I. Assem, D. Simson, A. Skowronski}, Elements of
Representation Theory of Associative Algebras, Volume 1: Techniques
of Representation Theory. LMS Student Texts (No. 65) (Cambridge
Univ. Press 2006).
\bibitem{aus} {\it M. Auslander, I. Reiten, S. Smalo},
Representation Theory of Artin Algebras, Cambridge Stud. in Adv.
Math. 36 (Cambridge Univ. Press 1995).
\bibitem{CalRei} {\it P. Caldero, M. Reineke}, On the quiver Grassmannians in the
acyclic case. Preprint arxiv math.RT/0611074
\bibitem{CalZel} {\it P. Caldero, A. Zelevinsky}, Laurent expansions in cluster algebras via quiver
representations. Moscow Mathematical Journal, Vol. 6, special issue
in honor of Alexander Alexandrovich Kirillov on the occasion of his
seventieth birthday, (2006), 411-429.
\bibitem{dlabrin} {\it V. Dlab, C. M. Ringel},
Indecomposable representations of graphs and algebras. AMS Memoirs
173 (1976).
\bibitem{Macd} {\it I. G. Macdonald},
Symmetric Functions and Hall Polynomials. Clarendon Press Oxford
1995.
\bibitem{Nan} {\it T. S. Nanjundiah}, Remark on a note of P.
Tur\'an. Am. Monthly 65(1958), 354.
\bibitem{rin1} {\it C. M. Ringel},
Tame algebras and Integral Quadratic Forms. Lect. Notes Math. 1099
(Springer 1984).
\bibitem{Szanto}{\it Cs. Sz\'ant\'o}, Hall numbers and the
composition algebra of the Kronecker algebra. Algebras and
Representation Theory 9,(2006), 465-495.

\end{thebibliography}
\end{document}